\documentclass[12pt,reqno]{article}

\usepackage[usenames]{color}
\usepackage{amssymb}
\usepackage{amscd}

\usepackage[colorlinks=true,
linkcolor=webgreen,
filecolor=webbrown,
citecolor=webgreen]{hyperref}

\definecolor{webgreen}{rgb}{0,.5,0}
\definecolor{webbrown}{rgb}{.6,0,0}

\usepackage{color}
\usepackage{fullpage}
\usepackage{float}

\usepackage{graphics,amsmath,amssymb}
\numberwithin{equation}{section}
\usepackage{amsthm}
\usepackage{amsfonts}
\usepackage{latexsym}

\DeclareMathOperator{\Li}{Li}

\begin{document}


\theoremstyle{plain}
\newtheorem{theorem}{Theorem}
\newtheorem{corollary}[theorem]{Corollary}
\newtheorem{lemma}[theorem]{Lemma}
\theoremstyle{remark}
\numberwithin{theorem}{section}
\newtheorem*{remark}{Remark}
\newtheorem*{example}{Examples}

\newcommand{\lrf}[1]{\left\lfloor #1\right\rfloor}
Last updated~\today .
\begin{center}
\vskip 1cm{\LARGE\bf Golden Ratio Base Expansions\\ of the Logarithm and Inverse Tangent\\
of Fibonacci and Lucas Numbers
\vskip .11in }

\vskip 1cm

{\large

\vskip 0.2 in

Kunle Adegoke \\
Department of Physics and Engineering Physics, \\ Obafemi Awolowo University, Ile-Ife\\ Nigeria \\
\href{mailto:adegoke00@gmail.com}{\tt adegoke00@gmail.com}

\vskip 0.2 in

Jaume Oliver Lafont  \\
Conselleria d'Educacio i Cultura, Govern de les Illes Balears,\\ Palma, Spain\\
\href{mailto:joliverlafont@gmail.com}{\tt joliverlafont@gmail.com}}

\end{center}

\vskip .2 in

\begin{abstract}
Let $\alpha=(1+\sqrt 5)/2$, the golden ratio, and $\beta=-1/\alpha=(1 - \sqrt 5)/2$. Let $F_n$ and $L_n$ be the Fibonacci and Lucas numbers, defined by $F_n=(\alpha^n -\beta^n)/\sqrt 5$ and $L_n=\alpha^n + \beta^n$, for all non-negative integers. We derive base~$\alpha$ expansions of $\log F_n$, $\log L_n$, $\arctan\dfrac1{F_n}$ and $\arctan\dfrac1{L_n}$ for all positive integers $n$.
\end{abstract}


\section{Introduction}
Let $\alpha$ denote the golden ratio; that is $\alpha=(1+\sqrt 5)/2$. Let $\beta=-1/\alpha=(1 - \sqrt 5)/2$. Thus $\alpha\beta=-1$ and $\alpha + \beta=1$. Let $F_n$ and $L_n$ be the Fibonacci and Lucas numbers, defined by $F_n=(\alpha^n -\beta^n)/\sqrt 5$ and $L_n=\alpha^n + \beta^n$, for all non-negative integers $n$.

Let $b$ be any non-zero number whose magnitude is greater than unity. Let $n$ and $s$ be positive integers. A convergent series of the form
\begin{equation}\label{eq.gngbiu8}
C = \sum_{k = 0}^\infty  {\frac{1}{{b^k }}\left( {\frac{{a_1 }}{{(kn + 1)^s }} + \frac{{a_2 }}{{(kn + 2)^s }} +  \cdots  + \frac{{a_n }}{{(kn + n)^s }}} \right)}, 
\end{equation}
where $a_1$, $a_2$, $\ldots$, $a_n$ are certain numbers, defines a base~$b$, length $n$ and degree $s$ expansion of the mathematical constant $C$.

If $b$ is an integer and $a_k$ are rational numbers, then~\eqref{eq.gngbiu8} is referred to as a BBP-type formula, after the initials of the authors of the paper~\cite{bbp97} in which such an expansion was first presented for $\pi$ and some other mathematical constants.  Any mathematical constant that possesses a base $b$ BBP-type formula has the property that its $n-$th digit in base $b$ could be calculated directly, without needing to compute any of the previous $n - 1$ digits. Although the infinite series presented in this paper have the structure of BBP-type formulas, it must be clearly stated that the series do not yield digit extraction since here the base $b=\alpha^n$ is not an integer and the $a_k$ are not rational numbers; rather the series correspond to base $\alpha$ expansions of the mathematical constants concerned.

Our goal in this paper is to derive base $\alpha$ expansion formulas for the logarithm and the inverse tangent of all Fibonacci and Lucas numbers. We will often give the expansion using the compact $P-$notation for BBP-type formulas, introduced by Bailey and Crandall~\cite{baileycrandall01}, namely,
\[
C = P(s,b,n,A) = \sum_{k = 0}^\infty  {\frac{1}{b^k }}\sum_{j = 1}^n {\frac{a_j }{{(kn + j)^s }} }, 
\]
where $s$ and $n$ are integers and, in this present paper, $b$ is an integer power of $\alpha$ and $A=(a_1,a_2,\ldots,a_n)$ is a vector of rational multiples of powers of $\beta$. For example, we will show (see~\eqref{eq.a4iqzr2}) that
\[
\log F_3  = \log 2 = \sum_{k = 0}^\infty  {\frac{1}{{\alpha ^{12k} }}\left( {\frac{{\beta ^2 }}{{6k + 1}} + \frac{{3\beta ^4 }}{{6k + 2}} + \frac{{4\beta ^6 }}{{6k + 3}} + \frac{{3\beta ^8 }}{{6k + 4}} + \frac{{\beta ^{10} }}{{6k + 5}}} \right)},
\]
which, in the $P-$notation, can be written as
\[
\log F_3  = \log 2 = P(1,\alpha ^{12} ,6,(\beta ^2 ,3\beta ^4 ,4\beta ^6 ,3\beta ^8 ,\beta ^{10} ,0)).
\]
Base~$\alpha$ expansions have also been studied or reported by Bailey and Crandall~\cite{baileycrandall01}, Chan~\cite{chan06,chan08}, Zhang~\cite{zhang13}, Borwein and Chamberland~\cite{borwein}, Cloitre~\cite{cloitre}, Adegoke~\cite{adegoke14}, Wei~\cite{wei15}, and more recently Kristensen and Mathiasen~\cite{KrisMath23}.

\section{Base $\alpha$ expansions of logarithms}
The base $\alpha$ expansions of the logarithms of Fibonacci and Lucas numbers are presented in Theorems~\ref{thm.ekxlm8t} and~\ref{thm.l3c4iuf} but first we state a couple of Lemmata upon which the results are based.

Let
\[
\Li_1 (x) =  - \log (1 - x) = \sum_{k = 1}^\infty  {\frac{x^k }{k}}  = \sum_{k = 0}^\infty  {x^k \frac{x}{k + 1}},\quad -1\le x<1. 
\]
\begin{lemma}
If $|b|>1$, $t>0$ and $m$ and $n$ are arbitrary positive integers, then,
\begin{equation}\label{eq.zelrhfk}
\Li_1 \left( {\frac{1}{b{\,}^t }} \right) = \sum_{k = 0}^\infty  {\frac{1}{{b^{tmk} }}\sum_{j = 1}^m {\frac{{1/b{\,}^{tj} }}{{(mk + j)}}} },
\end{equation}
\begin{equation}\label{eq.di8tmxj}
\Li_1 \left( -\frac{1}{{b{\,}^t }} \right) = \sum_{k = 0}^\infty  {\frac{1}{{b^{2tnk} }}\sum_{j = 1}^{2n} {\frac{{( - 1)^j /b{\,}^{tj} }}{{(2nk + j)}}} }.
\end{equation}
\end{lemma}
\begin{proof}
We have
\[
\Li_1 \left( {\frac{1}{{b{\,}^t }}} \right) = \sum_{k = 0}^\infty  {\frac{1}{{b{\,}^{tk} }}\frac{{1/b{\,}^t }}{{k + 1}}} ,
\]
from which~\eqref{eq.zelrhfk} follows upon using the identity
\begin{equation}\label{eq.fawa3l2}
\sum_{k = 0}^\infty  {f_k }  = \sum_{k = 0}^\infty  {\sum_{j = 1}^m {f_{mk + j - 1} } },
\end{equation}
with
\[
f_k  = \frac{1}{{b^{tk + t} }}\frac{1}{{k + 1}}.
\]
The proof of \eqref{eq.di8tmxj} is similar, with $m=2n$ in~\eqref{eq.fawa3l2}.
\end{proof}
\begin{lemma}
If $r$ is an integer, then,
\begin{equation}\label{eq.wtneqxp}
\log L_r  = r\Li_1 \left( {\frac{1}{{\alpha ^2 }}} \right) - \Li_1 \left( {\frac{{( - 1)^{r + 1} }}{{\alpha ^{2r} }}} \right),
\end{equation}

\begin{equation}\label{eq.hfkwjtp}
\log F_r  = (r - 2)\Li_1 \left( {\frac{1}{{\alpha ^2 }}} \right) + \Li_1 \left( {\frac{1}{{\alpha ^4 }}} \right) - \Li_1 \left( {\frac{{( - 1)^r }}{{\alpha ^{2r} }}} \right),\quad r\ne 0.
\end{equation}

\end{lemma}
\begin{proof}
We have
\begin{equation}\label{eq.vpz5mwv}
\Li_1 \left( { - \frac{{\beta ^r }}{{\alpha ^r }}} \right) =  - \log \left( {\frac{{\alpha ^r  + \beta ^r }}{{\alpha ^r }}} \right) =  - \log \left( {\frac{{L_r }}{{\alpha ^r }}} \right) =  - \log L_r  + r\log \alpha,
\end{equation}
in which setting $r=1$ gives
\begin{equation}\label{eq.hjknap2}
\log \alpha  = \Li_1 \left( {\frac{1}{\alpha ^2 }} \right).
\end{equation}
Using~\eqref{eq.hjknap2} in~\eqref{eq.vpz5mwv} gives~\eqref{eq.wtneqxp}.

Also,
\begin{equation}\label{eq.i0fkkqq}
\Li_1 \left( {\frac{{\beta ^r }}{{\alpha ^r }}} \right) =  - \log \left( {\frac{{\alpha ^r  - \beta ^r }}{{\alpha ^r }}} \right) =  - \log \left( {\frac{{F_r \sqrt 5 }}{{\alpha ^r }}} \right) =  - \log F_r  + r\log \alpha  - \log \sqrt 5, 
\end{equation}
in which setting $r=2$ gives
\begin{equation}\label{eq.lnj856k}
\log \sqrt 5  = 2\log \alpha  - \Li_1 \left( {\frac{1}{{\alpha ^4 }}} \right) = 2\Li_1 \left( {\frac{1}{{\alpha ^2 }}} \right) - \Li_1 \left( {\frac{1}{{\alpha ^4 }}} \right),
\end{equation}
where we used~\eqref{eq.hjknap2}. Identity~\eqref{eq.hfkwjtp} follows from~\eqref{eq.i0fkkqq} and~\eqref{eq.lnj856k}.

\end{proof}
\begin{theorem}\label{thm.ekxlm8t}
If $r$ is an integer, then,
\begin{equation}\label{eq.cld23eo}
\log F_r  = \sum_{k = 0}^\infty  {\frac{1}{{\alpha ^{4rk} }}\sum_{j = 1}^r {\frac{{\beta ^{4j - 2} (r - 2 + r\delta _{j,(r + 1)/2} )}}{{2rk + 2j - 1}}} }  + \sum_{k = 0}^\infty  {\frac{1}{{\alpha ^{4rk} }}\sum_{j = 1}^{r - 1} {\frac{{\beta ^{4j} r}}{{2rk + 2j}}} }, \quad\mbox{$r$ odd},
\end{equation}
\begin{equation}\label{eq.fqg0g5o}
\log F_r  = \sum_{k = 0}^\infty  {\frac{1}{{\alpha ^{4rk} }}\sum_{j = 1}^r {\frac{{(r - 2)\beta ^{4j - 2} }}{{2rk + 2j - 1}}} }  + \sum_{k = 0}^\infty  {\frac{1}{{\alpha ^{4rk} }}\sum_{j = 1}^{r - 1} {\frac{{\beta ^{4j} r(1 - \delta _{j,r/2} )}}{{2rk + 2j}}} }, \quad\mbox{$r$ even}.
\end{equation}
\end{theorem}
Here and throughout this paper, $\delta_{mn}$ denotes the Kronecker delta symbol whose value is unity when $m$ equals $n$ and zero otherwise.
\begin{proof}
We prove~\eqref{eq.fqg0g5o}. When $r$ is even,~\eqref{eq.hfkwjtp} reads
\begin{equation}\label{eq.fb8sqed}
\log F_r  = (r - 2)\Li_1 \left( {\frac{1}{{\alpha ^2 }}} \right) + \Li_1 \left( {\frac{1}{{\alpha ^4 }}} \right) - \Li_1 \left( {\frac{1}{{\alpha ^{2r} }}} \right),\quad r\ne 0.
\end{equation}
We proceed to write the three $\Li_1$ terms in a common base $\alpha^{4r}$, using~\eqref{eq.zelrhfk} with appropriate $t$ and $m$ choices. Thus,
\begin{equation}\label{eq.zg3hp39}
\Li_1 \left( {\frac{1}{{\alpha ^2 }}} \right) = \sum_{k = 0}^\infty  {\frac{1}{{\alpha ^{4rk} }}\sum_{j = 1}^{2r} {\frac{{1/\alpha ^{2j} }}{{2rk + j}}} } ,
\end{equation}
\begin{equation}\label{eq.vuc0s2z}
\Li_1 \left( {\frac{1}{{\alpha ^4 }}} \right) = \sum_{k = 0}^\infty  {\frac{1}{{\alpha ^{4rk} }}\sum_{j = 1}^r {\frac{{1/\alpha ^{4j} }}{{rk + j}}} } ,
\end{equation}
\begin{equation}\label{eq.o00vv93}
\Li_1 \left( {\frac{1}{{\alpha ^{2r} }}} \right) = \sum_{k = 0}^\infty  {\frac{1}{{\alpha ^{4rk} }}\left( {\frac{{1/\alpha ^{2r} }}{{2k + 1}} + \frac{{1/\alpha ^{4r} }}{{2k + 2}}} \right)} .
\end{equation}
Using~\eqref{eq.zg3hp39},~\eqref{eq.vuc0s2z} and~\eqref{eq.o00vv93} in~\eqref{eq.fb8sqed} gives
\begin{equation}\label{eq.u1h85k9}
\begin{split}
\log F_r  &= \sum_{k = 0}^\infty  {\frac{1}{{\alpha ^{4rk} }}\sum_{j = 1}^{2r} {\frac{{\beta ^{2j}(r - 2) }}{{2rk + j}}} }  + \sum_{k = 0}^\infty  {\frac{1}{{\alpha ^{4rk} }}\sum_{j = 1}^r {\frac{{2\beta ^{4j} }}{{2rk + 2j}}} } \\
&\qquad - \sum_{k = 0}^\infty  {\frac{1}{{\alpha ^{4rk} }}\left( {\frac{{r\beta ^{2r} }}{{2rk + r}} + \frac{{r\beta ^{4r} }}{{2rk + 2r}}} \right)} .
\end{split}
\end{equation}
Using the summation identity
\begin{equation}
\sum_{j = 1}^{2r} {f_j }  = \sum_{j = 1}^r {f_{2j} }  + \sum_{j = 1}^r {f_{2j - 1} }
\end{equation}
to write its inner sum, the first term on the right hand side of~\eqref{eq.u1h85k9} can be written as
\begin{equation}\label{eq.zll0uaz}
\sum_{k = 0}^\infty  {\frac{1}{{\alpha ^{4rk} }}\sum_{j = 1}^{2r} {\frac{{\beta ^{2j} (r - 2)}}{{2rk + j}}} }  = \sum_{k = 0}^\infty  {\frac{1}{{\alpha ^{4rk} }}\sum_{j = 1}^r {\frac{{\beta ^{4j} (r - 2)}}{{2rk + 2j}}} }  + \sum_{k = 0}^\infty  {\frac{1}{{\alpha ^{4rk} }}\sum_{j = 1}^r {\frac{{\beta ^{4j - 2} (r - 2)}}{{2rk + 2j - 1}}} }.
\end{equation}
Using~\eqref{eq.zll0uaz} in~\eqref{eq.u1h85k9} yields~\eqref{eq.fqg0g5o}.

\end{proof}
Identities~\eqref{eq.cld23eo} and~\eqref{eq.fqg0g5o} written in the $P-$notation are
\begin{equation}
\log F_r  = P(1,\alpha ^{4r} ,2r,(a_1 ,a_2 , \ldots ,a_{2r} ))
\end{equation}
where for $1\le j\le r$,
\[
a_{2j - 1}  = \beta ^{4j - 2} (r - 2 + r\delta _{j,(r + 1)/2} ),\quad a_{2j}  = \beta ^{4j} r(1 - \delta _{rj} ),\quad\mbox{$r$ odd};
\]
and
\[
a_{2j - 1}  = (r - 2)\beta ^{4j - 2} ,\quad a_{2j}  = \beta ^{4j} r(1 - \delta _{j,r/2}  - \delta _{j,r} ),\quad\mbox{$r$ even}.
\]
\begin{example}
\begin{equation}\label{eq.a4iqzr2}
\log F_3  = \log 2 = P(1,\alpha ^{12} ,6,(\beta ^2 ,3\beta ^4 ,4\beta ^6 ,3\beta ^8 ,\beta ^{10} ,0)),
\end{equation}
\begin{equation}\label{eq.iw04dtw}
\begin{split}
\log F_5  = \log 5 &= P(1,\alpha ^{20} ,10,(3\beta ^2 ,5\beta ^4 ,3\beta ^6 ,5\beta ^8 ,8\beta ^{10} ,\\
&\qquad 5\beta ^{12} ,3\beta ^{14} ,5\beta ^{16} ,3\beta ^{18},0 )),
\end{split}
\end{equation}
\begin{equation}\label{eq.osuk393}
\log F_4  = \log 3 = P(1,\alpha ^{16} ,8,(2\beta ^2 ,4\beta ^4 ,2\beta ^6 ,0,2\beta ^{10} ,4\beta ^{12} ,2\beta ^{14} ,0)),
\end{equation}
\begin{equation}
\begin{split}
\log F_8  = \log 21 &= P(1,\alpha ^{32} ,16,(6\beta ^2 ,8\beta ^4 ,6\beta ^6 ,8\beta ^8 ,6\beta ^{10} ,8\beta ^{12} ,6\beta ^{14} ,\\
&\qquad 0,6\beta ^{18} ,8\beta ^{20} ,6\beta ^{22} ,8\beta ^{24} ,6\beta ^{26} ,8\beta ^{28} ,6\beta ^{30},0 )),
\end{split}
\end{equation}
\begin{equation}\label{eq.jz4ufni}
\begin{split}
\log F_{12}=\log 144&=P(1,\alpha^{48},24,(10\,\beta^{2},12\,\beta^{4},10\,\beta^{6}, 12\,\beta^{8},10\,\beta^{10},\\
&\qquad12\,\beta^{12},10\,\beta^{14},12\,\beta^{16},10\,\beta^{18},12\,\beta^{20},10\,\beta^{22},0
,\\
&\qquad\,10\,\beta^{26},12\,\beta^{28},10\,\beta^{30},12\,\beta^{32},10\,\beta^{34},12\,\beta^{36},\\
&\qquad\;10\,\beta^{38},12\,\beta^{40},10\,\beta^{42},12\,\beta^{44},10\,\beta^{
46},0)).
\end{split}
\end{equation}

\end{example}

\begin{theorem}\label{thm.l3c4iuf}
If $r$ is an integer, then,
\begin{equation}\label{eq.noc1ibg}
\log L_r  = \sum_{k = 0}^\infty  {\frac{1}{{\alpha ^{2rk} }}\sum_{j = 1}^{r - 1} {\frac{{\beta ^{2j} r}}{{rk + j}}} },  \mbox{ $r$ odd},
\end{equation}
\begin{equation}\label{eq.beeuajp}
\log L_r  = \sum_{k = 0}^\infty  {\frac{1}{{\alpha ^{4rk} }}\sum_{j = 1}^{2r - 1} {\frac{{\beta ^{2j} r(1 + \delta _{rj} )}}{{2rk + j}}} }, \mbox{ $r$ even}.
\end{equation}
\end{theorem}
\begin{proof}
We prove~\eqref{eq.noc1ibg}. If $r$ is an odd integer,~\eqref{eq.wtneqxp} gives
\begin{equation}\label{eq.p8eomy9}
\log L_r  = r\Li_1 \left( {\frac{1}{{\alpha ^2 }}} \right) - \Li_1 \left( {\frac{1}{{\alpha ^{2r} }}} \right).
\end{equation}
With
\[
\Li_1 \left( {\frac{1}{{\alpha ^2 }}} \right) = \sum_{k = 0}^\infty  {\frac{1}{{\alpha ^{2rk} }}\sum_{j = 1}^r {\frac{{1/\alpha ^{2j} }}{{rk + j}}} } 
\]
and
\[
\Li_1 \left( {\frac{1}{{\alpha ^{2r} }}} \right) = \sum_{k = 0}^\infty  {\frac{1}{{\alpha ^{2rk} }}\frac{{r/\alpha ^{2r} }}{{rk + r}}} 
\]
in~\eqref{eq.p8eomy9}; identity~\eqref{eq.noc1ibg} follows.
\end{proof}
Identities~\eqref{eq.noc1ibg} and~\eqref{eq.beeuajp} in the $P-$notation are
\begin{equation}
\log L_r  = P(1,\alpha ^{2r} ,r,(a_1 ,a_2 , \ldots ,a_r )),\quad\mbox{$r$ odd},
\end{equation}
with
\[
a_j  = r\beta ^{2j} (1 - \delta _{rj} ),\quad 1\le j\le r,
\]
and
\begin{equation}
\log L_r  = P(1,\alpha ^{4r} ,2r,(a_1 ,a_2 , \ldots ,a_{2r} )),\quad\mbox{$r$ even},
\end{equation}
with
\[
a_j=r\beta^{2j}(1 + \delta_{jr} -\delta_{j,2r}),\quad 1\le j\le 2r.
\]
\begin{example}
\begin{equation}
\log L_2=\log 3 = \sum_{k = 0}^\infty  {\frac{1}{{\alpha ^{8k} }}\left( {\frac{{2\beta ^2 }}{{4k + 1}} + \frac{{4\beta ^4 }}{{4k + 2}} + \frac{{2\beta ^6 }}{{4k + 3}}} \right)};
\end{equation}
that is,
\begin{equation}\label{eq.exslg0n}
\log 3 = P(1,\alpha ^8 ,4,(2\beta^2,4\beta ^4 ,2\beta ^6 ,0)).
\end{equation}
\begin{equation}\label{eq.qhwege5}
\log L_3  = \log 4 = P(1,\alpha ^6 ,3,(3\beta ^2 ,3\beta ^4 ,0)),
\end{equation}
\begin{equation}
\log L_4=\log 7 = 4\beta ^2 P(1,\alpha ^{16} ,8,(1,\beta ^2 ,\beta ^4 ,2\beta ^6 ,\beta ^8 ,\beta ^{10} ,\beta ^{12} ,0)),
\end{equation}
\begin{equation}\label{eq.tncj4ik}
\begin{split}
\log L_6  = \log 18 &= P(1,\alpha ^{24} ,12,(6\beta ^2 ,6\beta ^4 ,6\beta ^6 ,6\beta ^8 ,6\beta ^{10} ,\\
&\qquad12\beta ^{12} ,6\beta ^{14} ,6\beta^{16}, 6\beta ^{18} ,6\beta ^{20} ,6\beta ^{22} ,0)).
\end{split}
\end{equation}
\end{example}
\section{Base $\alpha$ expansions of inverse tangents}
The base~$\alpha$ expansions of the inverse tangent of Fibonacci and Lucas numbers are stated in Theorems~\ref{thm.aucw05k}--\ref{thm.qiyqzuj} but first we collect some required identities in Lemmata~\ref{lem.v3gato7}--\ref{lem.suxk4t4}.
\begin{lemma}\label{lem.v3gato7}
If $r$ is an integer, then,
\begin{equation}
\alpha ^r  - \alpha ^{ - r}  = \left\{ \begin{array}{l}
 F_r \sqrt 5,\quad\mbox{$r$ even};  \\ 
 L_r, \quad\mbox{$r$ odd} ;\\ 
 \end{array} \right.
\end{equation}
\begin{equation}
\alpha ^r  + \alpha ^{ - r}  = \left\{ \begin{array}{l}
 L_r\quad\mbox{$r$ even};  \\ 
 F_r \sqrt 5,\quad\mbox{$r$ odd}.  \\ 
 \end{array} \right.
\end{equation}
\end{lemma}
\begin{lemma}\label{lem.bzswbnr}
If $r$ and $m$ are integers, then,
\begin{equation}\label{eq.sbonzfq}
\arctan \frac{1}{{\alpha ^{r - m} }} - \arctan \frac{1}{{\alpha ^{r + m} }} = \left\{ \begin{array}{l}
 \arctan \left( {L_m /(F_r \sqrt 5 )} \right),\quad\mbox{$m$ odd, $r$ odd}, \\ 
 \arctan (L_m /L_r ),\quad\mbox{$m$ odd, $r$ even}, \\ 
 \arctan (F_m /F_r ),\quad\mbox{$m$ evev, $r$ odd}, \\ 
 \arctan (F_m \sqrt 5 /L_r ),\quad\mbox{$m$ even, $r$ even}; \\ 
 \end{array} \right.
\end{equation}

\begin{equation}\label{eq.goclmo5}
\arctan \frac{1}{{\alpha ^{r - m} }} + \arctan \frac{1}{{\alpha ^{r + m} }} = \left\{ \begin{array}{l}
 \arctan (F_m \sqrt 5 /L_r ),\quad\mbox{$m$ odd, $r$ odd}, \\ 
 \arctan (F_m /F_r ),\quad\mbox{$m$ odd, $r$ even}, \\ 
 \arctan (L_m /L_r ),\quad\mbox{$m$ even, $r$ odd}, \\ 
 \arctan \left( {L_m /(F_r \sqrt 5 )} \right),\quad\mbox{$m$ even, $r$ even} .\\ 
 \end{array} \right.
\end{equation}
\end{lemma}
\begin{proof}
The arctangent subtraction and addition formulas give
\[
\arctan \frac{1}{{\alpha ^{r - m} }} - \arctan \frac{1}{{\alpha ^{r + m} }} = \arctan \left( {\frac{{\alpha ^r (\alpha ^m  - \alpha ^{ - m} )}}{{\alpha ^{2r}  + 1}}} \right),
\]
\[
\arctan \frac{1}{{\alpha ^{r - m} }} + \arctan \frac{1}{{\alpha ^{r + m} }} = \arctan \left( {\frac{{\alpha ^r (\alpha ^m  + \alpha ^{ - m} )}}{{\alpha ^{2r}  - 1}}} \right);
\]
and hence the stated identities upon the use of Lemma~\ref{lem.v3gato7}.
\end{proof}
\begin{lemma}\label{lem.suxk4t4}
If $r$ is an integer, then,
\begin{gather}
\alpha ^{2r}  - 1 = \alpha ^r L_r ,\quad\beta ^{2r}  - 1 = \beta ^r L_r,\quad\mbox{$r$ odd},\label{eq.sbgg38c} \\
\alpha ^{2r}  - 1 = \alpha ^r F_r \sqrt 5 ,\quad\beta ^{2r}  - 1 =  - \beta ^r F_r \sqrt 5,\quad\mbox{$r$ even},\label{eq.yr9dvoi} \\
\alpha ^{2r}  + 1 = \alpha ^r F_r \sqrt 5 ,\quad\beta ^{2r}  + 1 =  - \beta ^r F_r \sqrt 5,\quad\mbox{$r$ odd}, \\
\alpha ^{2r}  + 1 = \alpha ^r L_r ,\quad\beta ^{2r}  + 1 = \beta ^r L_r,\quad\mbox{$r$ even}. 
\end{gather}
\end{lemma}
\begin{theorem}\label{thm.aucw05k}
If $r$ is an odd integer greater than unity, then,
\begin{equation}\label{eq.tgr62sm}
\arctan \frac1{{F_r }} = P(1,\alpha ^{4(r^2  - 4)} ,4(r^2  - 4),(a_1 ,a_2 , \ldots ,a_{4(r^2  - 4)} )),
\end{equation}
where the only non-zero constants $a_j$ are given by
\[
a_{(r - 2)(4j - 3)}  =  - \beta ^{(r - 2)(4j - 3)} (r - 2),\quad j = 1,2, \ldots ,r + 2,
\]
\[
a_{(r - 2)(4j - 1)}  = \beta ^{(r - 2)(4j - 1)} (r - 2),\quad j = 1,2, \ldots ,r + 2,
\]
\[
a_{(r + 2)(4j - 3)}  = \beta ^{(r + 2)(4j - 3)} (r + 2),\quad j = 1,2, \ldots ,r - 2,
\]
\[
a_{(r + 2)(4j - 1)}  =  - \beta ^{(r + 2)(4j - 1)} (r + 2),\quad j = 1,2, \ldots ,r - 2,
\]
\[
a_{(r - 2)(r + 2)}  = ( - 1)^{(r + 1)/2} 4\beta ^{r^2  - 4} 
\]
and
\[
a_{3(r - 2)(r + 2)}  = ( - 1)^{(r - 1)/2} 4\beta ^{3(r^2  - 4)}. 
\]

\end{theorem}
\begin{proof}
With $r$ an odd number, $m=2$ in~\eqref{eq.sbonzfq} gives
\[
\arctan \frac1{F_r } = \arctan \frac1{{\alpha ^{r - 2} }} - \arctan \frac1{{\alpha ^{r + 2} }}.
\]
The following identity, proved in~\cite[Identity (10)]{adegoke10}:
\begin{equation}\label{eq.ubd4g1u}
n\sqrt n \arctan \left( {\frac{1}{{\sqrt n }}} \right) = \sum_{k = 0}^\infty  {\frac{1}{{n^{2k} }}\left( {\frac n{{4k + 1}} - \frac{1}{{4k + 3}}} \right)}
\end{equation}
gives
\[
\arctan \frac{1}{{\alpha ^{r - 2} }} = \sum_{k = 0}^\infty  {\frac{1}{{\alpha ^{(r - 2)(4k + 3)} }}\left( {\frac{{\alpha ^{2r - 4} }}{{4k + 1}} - \frac{1}{{4k + 3}}} \right)} 
\]
and
\[
\arctan \frac{1}{{\alpha ^{r + 2} }} = \sum_{k = 0}^\infty  {\frac{1}{{\alpha ^{(r + 2)(4k + 3)} }}\left( {\frac{{\alpha ^{2r + 4} }}{{4k + 1}} - \frac{1}{{4k + 3}}} \right)}, 
\]
or, by~\eqref{eq.fawa3l2},
\[
\arctan \frac{1}{{\alpha ^{r - 2} }} = \sum_{k = 0}^\infty  {\frac{1}{{\alpha ^{(4r^2  - 16)k} }}\sum_{j = 1}^{r + 2} {\left( {\frac{{\alpha ^{ - (r - 2)(4j - 3)} }}{{4(r + 2)k + 4j - 3}} - \frac{{\alpha ^{ - (r - 2)(4j - 1)} }}{{4(r + 2)k + 4j - 1}}} \right)} } 
\]
and
\[
\arctan \frac{1}{{\alpha ^{r + 2} }} = \sum_{k = 0}^\infty  {\frac{1}{{\alpha ^{(4r^2  - 16)k} }}\sum_{j = 1}^{r - 2} {\left( {\frac{{\alpha ^{ - (r + 2)(4j - 3)} }}{{4(r - 2)k + 4j - 3}} - \frac{{\alpha ^{ - (r + 2)(4j - 1)} }}{{4(r - 2)k + 4j - 1}}} \right)} } .
\]
Thus,
\begin{equation}\label{eq.g92low5}
\begin{split}
\arctan \frac{1}{{F_r }} &= \sum_{k = 0}^\infty  {\frac{1}{{\alpha ^{(4r^2  - 16)k} }}\; \times } \\
&\qquad\qquad\left(\; {\sum_{j = 1}^{r + 2} {\left( {\frac{{ - \beta ^{(r - 2)(4j - 3)} (r - 2)}}{{4(r^2  - 4)k + (r - 2)(4j - 3)}} + \frac{{\beta ^{(r - 2)(4j - 1)} (r - 2)}}{{4(r^2  - 4)k + (r - 2)(4j - 1)}}} \right)} } \right.\\
&\qquad\qquad\quad\left. { + \sum_{j = 1}^{r - 2} {\left( {\frac{{\beta ^{(r + 2)(4j - 3)} (r + 2)}}{{4(r^2  - 4)k + (r + 2)(4j - 3)}} - \frac{{\beta ^{(r + 2)(4j - 1)} (r + 2)}}{{4(r^2  - 4)k + (r + 2)(4j - 1)}}} \right)} } \right).
\end{split}
\end{equation}
Identity~\eqref{eq.tgr62sm} is~\eqref{eq.g92low5} expressed in the $P-$notation.
\end{proof}
\begin{example}
\begin{equation}\label{eq.q0nwkrz}
\begin{split}
\arctan\frac1{F_3}=\arctan\frac12&=P ( 1,{\alpha}^{20},20,(-\beta,0,{\beta}^{3},0,4\,{\beta}^{5},0,
{\beta}^{7},0,-{\beta}^{9},0,{\beta}^{11},\\
&\qquad\qquad0,-{\beta}^{13},0,-4\,{\beta}^{15},0,-{\beta}^{17},0,{\beta}^{19},0)),
\end{split}
\end{equation}
\begin{equation}
\begin{split}
\arctan\frac1{F_5}=\arctan\frac15&=P ( 1,{\alpha}^{84},84,(0,0,-3\,{\beta}^{3},0,0,0,7\,{\beta}^{7}
,0,3\,{\beta}^{9},0,0,0,0,0,-3\,{\beta}^{15},\\
&\qquad\qquad0,0,0,0,0,-4\,{\beta}^{21
},0,0,0,0,0,-3\,{\beta}^{27},0,0,0,0,0,3\,{\beta}^{33},\\
&\qquad\qquad\,0,7\,{\beta}^{
35},0,0,0,-3\,{\beta}^{39},0,0,0,0,0,3\,{\beta}^{45},0,0,0,-7\,{\beta}
^{49},\\
&\qquad\qquad\;0,-3\,{\beta}^{51},0,0,0,0,0,3\,{\beta}^{57},0,0,0,0,0,4\,{\beta
}^{63},0,0,0,0,0,\\
&\qquad\qquad\;\,3\,{\beta}^{69},0,0,0,0,0,-3\,{\beta}^{75},0,-7\,{
\beta}^{77},0,0,0,3\,{\beta}^{81},0,0,0) ),
\end{split}
\end{equation}
\begin{equation}
\begin{split}
\arctan\frac1{F_7}=\arctan\frac1{13}&=P ( 1,{\alpha}^{180},180,(0,0,0,0,-5\,{\beta}^{5},0,0,0,9\,{
\beta}^{9},0,0,0,0,0,5\,{\beta}^{15},\\
&\quad\qquad0,0,0,0,0,0,0,0,0,-5\,{\beta}^{25
},0,-9\,{\beta}^{27},0,0,0,0,0,0,0,\\
&\,\qquad\qquad5\,{\beta}^{35},0,0,0,0,0,0,0,0,0,4
\,{\beta}^{45},0,0,0,0,0,0,0,0,0,\\
&\;\qquad\qquad5\,{\beta}^{55},0,0,0,0,0,0,0,-9\,{
\beta}^{63},0,-5\,{\beta}^{65},0,0,0,0,0,\\
&\,\;\qquad\qquad0,0,0,0,5\,{\beta}^{75},0,0,0
,0,0,9\,{\beta}^{81},0,0,0,-5\,{\beta}^{85},0,0,\\
&\;\;\qquad\qquad0,0,0,0,0,0,0,5\,{
\beta}^{95},0,0,0,-9\,{\beta}^{99},0,0,0,0,0,\\
&\;\;\,\qquad\qquad-5\,{\beta}^{105},0,0,0,0
,0,0,0,0,0,5\,{\beta}^{115},0,9\,{\beta}^{117},0,0,0,\\
&\;\;\;\qquad\qquad0,0,0,0,-5\,{
\beta}^{125},0,0,0,0,0,0,0,0,0,-4\,{\beta}^{135},\\
&\;\;\;\,\qquad\qquad0,0,0,0,0,0,0,0,0,-5
\,{\beta}^{145},0,0,0,0,0,0,0,9\,{\beta}^{153},0,\\
&\;\;\;\;\qquad\qquad5\,{\beta}^{155},0,0,0
,0,0,0,0,0,0,-5\,{\beta}^{165},0,0,0,0,0,-9\,{\beta}^{171},\\
&\;\;\;\;\,\qquad\qquad0,0,0,5\,{
\beta}^{175},0,0,0,0,0) ).
\end{split}
\end{equation}

\end{example}
\begin{theorem}\label{thm.piaqnr3}
If $r$ is a positive even integer, then,
\begin{equation}\label{eq.vkh5mhm}
\arctan \frac1{{F_r }} = P(1,\alpha ^{4(r^2  - 1)} ,4(r^2  - 1),(a_1 ,a_2 , \ldots ,a_{4(r^2  - 1)} )),
\end{equation}
where the only non-zero constants $a_j$ are given by
\[
a_{(r - 1)(4j - 3)}  =  -\beta ^{(r - 1)(4j - 3)} (r - 1),\quad j = 1,2, \ldots ,r + 1,
\]
\[
a_{(r - 1)(4j - 1)}  = \beta ^{(r - 1)(4j - 1)} (r - 1),\quad j = 1,2, \ldots ,r + 1,
\]
\[
a_{(r + 1)(4j - 3)}  = -\beta ^{(r + 1)(4j - 3)} (r + 1),\quad j = 1,2, \ldots ,r - 1,
\]
\[
a_{(r + 1)(4j - 1)}  =   \beta ^{(r + 1)(4j - 1)} (r + 1),\quad j = 1,2, \ldots ,r - 1,
\]
\[
a_{(r - 1)(r + 1)}  = ( - 1)^{r/2} 2\beta ^{r^2  - 1} 
\]
and
\[
a_{3(r - 1)(r + 1)}  = ( - 1)^{(r + 2)/2} 2\beta ^{3(r^2  - 1)}. 
\]

\end{theorem}
\begin{proof}
Setting $m=1$ in~\eqref{eq.goclmo5} gives
\[
\arctan \frac1{F_r } = \arctan \frac1{{\alpha ^{r - 1} }} + \arctan \frac1{{\alpha ^{r + 1} }},\quad\mbox{$r$ even}.
\]
The proof now proceeds as in that of Theorem~\ref{thm.aucw05k}.
\end{proof}
\begin{example}
\begin{equation}
\arctan\frac1{F_2}=\frac\pi 4=P(1,{\alpha}^{12},12,(-\beta,0,-2\,{\beta}^{3},0,-{\beta}^{5},0,{\beta}
^{7},0,2\,{\beta}^{9},0,{\beta}^{11},0)),
\end{equation}
\begin{equation}\label{eq.u6gman9}
\begin{split}
\arctan\frac1{F_4}=\arctan\frac13&=P(1,{\alpha}^{60},60,(0,0,-3\,{\beta}^{3},0,-5\,{\beta}^{5},0,0,0,3\,{
\beta}^{9},0,0,0,0,0,2\,{\beta}^{15},\\
&\qquad\qquad0,0,0,0,0,3\,{\beta}^{21},0,0,0,-
5\,{\beta}^{25},0,-3\,{\beta}^{27},0,0,0,0,0,3\,{\beta}^{33},\\
&\qquad\qquad\;0,5\,{
\beta}^{35},0,0,0,-3\,{\beta}^{39},0,0,0,0,0,-2\,{\beta}^{45},0,0,0,0,0
,-3\,{\beta}^{51},\\
&\qquad\qquad\;\;0,0,0,5\,{\beta}^{55},0,3\,{\beta}^{57},0,0,0)).
\end{split}
\end{equation}
\end{example}
\begin{theorem}\label{thm.i0pt0t0}
If $r$ is a positive even integer, then,
\begin{equation}\label{eq.d7aeveo}
\arctan \frac1{{L_r }} = P(1,\alpha ^{4(r^2  - 1)} ,4(r^2  - 1),(a_1 ,a_2 , \ldots ,a_{4(r^2  - 1)} )),
\end{equation}
where the only non-zero constants $a_j$ are given by
\[
a_{(r - 1)(4j - 3)}  =  -\beta ^{(r - 1)(4j - 3)} (r - 1),\quad j = 1,2, \ldots ,r + 1,
\]
\[
a_{(r - 1)(4j - 1)}  = \beta ^{(r - 1)(4j - 1)} (r - 1),\quad j = 1,2, \ldots ,r + 1,
\]
\[
a_{(r + 1)(4j - 3)}  = \beta ^{(r + 1)(4j - 3)} (r + 1),\quad j = 1,2, \ldots ,r - 1,
\]
\[
a_{(r + 1)(4j - 1)}  =  -\beta ^{(r + 1)(4j - 1)} (r + 1),\quad j = 1,2, \ldots ,r - 1,
\]
\[
a_{(r - 1)(r + 1)}  = ( - 1)^{(r + 2)/2}\, 2r\,\beta ^{r^2  - 1} 
\]
and
\[
a_{3(r - 1)(r + 1)}  = ( - 1)^{r/2}\, 2r\,\beta ^{3(r^2  - 1)}. 
\]

\end{theorem}
\begin{proof}
Setting $m=1$ in~\eqref{eq.sbonzfq} gives
\[
\arctan \frac1{L_r } = \arctan \frac1{{\alpha ^{r - 1} }} - \arctan \frac1{{\alpha ^{r + 1} }},\quad\mbox{$r$ even}.
\]
The proof now proceeds as in that of Theorem~\ref{thm.aucw05k}.
\end{proof}
\begin{example}
\begin{equation}\label{eq.blxa1pt}
\arctan\frac1{L_2}=\arctan\frac13=(1,{\alpha}^{12},12,(-\beta,0,4\,{\beta}^{3},0,-{\beta}^{5},0,{\beta}^
{7},0,-4\,{\beta}^{9},0,{\beta}^{11},0)),
\end{equation}
\begin{equation}
\begin{split}
\arctan\frac1{L_4}=\arctan\frac17&=(1,{\alpha}^{60},60,(0,0,-3\,{\beta}^{3},0,5\,{\beta}^{5},0,0,0,3\,{
\beta}^{9},0,0,0,0,0,-8\,{\beta}^{15},\\
&\qquad\qquad0,0,0,0,0,3\,{\beta}^{21},0,0,0,
5\,{\beta}^{25},0,-3\,{\beta}^{27},0,0,0,0,0,3\,{\beta}^{33},\\
&\qquad\qquad\;0,-5\,{
\beta}^{35},0,0,0,-3\,{\beta}^{39},0,0,0,0,0,8\,{\beta}^{45},0,0,0,0,0
,\\
&\qquad\qquad\;\;-3\,{\beta}^{51},0,0,0,-5\,{\beta}^{55},0,3\,{\beta}^{57},0,0,0)).
\end{split}
\end{equation}
\end{example}
\begin{theorem}
If $r$ is an integer, then,
\begin{equation}\label{eq.ulfkx09}
\begin{split}
&\sum_{k = 0}^\infty  {\frac{1}{{\alpha ^{12rk} }}\left( {\frac{{\beta ^r }}{{12k + 1}} + \frac{{2\beta ^{3r} }}{{12k + 3}} + \frac{{\beta ^{5r} }}{{12k + 5}} - \frac{{\beta ^{7r} }}{{12k + 7}} - \frac{{2\beta ^{9r} }}{{12k + 9}} - \frac{{\beta ^{11r} }}{{12k + 11}}} \right)} \\
&\qquad= \left\{ \begin{array}{l}
  - \arctan \left( {\frac{1}{{L_r }}} \right),\quad\mbox{$r$ odd}, \\ 
 \arctan \left( {\frac{1}{{F_r \sqrt 5 }}} \right),\quad\mbox{$r$ even}; \\ 
 \end{array} \right.
\end{split}
\end{equation}
\end{theorem}
that is,
\begin{equation}
\begin{split}
&P(1,\alpha ^{12r} ,12,(\beta ^r ,0,2\beta ^{3r} ,0,\beta ^{5r} ,0, - \beta ^{7r} ,0, - 2\beta ^{9r} ,0, - \beta ^{11r} ,0))\\
&\qquad= \left\{ \begin{array}{l}
  - \arctan \left( {\frac{1}{{L_r }}} \right),\quad\mbox{$r$ odd}, \\ 
 \arctan \left( {\frac{1}{{F_r \sqrt 5 }}} \right),\quad\mbox{$r$ even}. \\ 
 \end{array} \right.
\end{split}
\end{equation}
\begin{proof}
In~\cite[Identity (27)]{adegoke10}, it was shown that
\[
n^2 \sqrt n \arctan \left( {\frac{{\sqrt n }}{n - 1}} \right) = \sum_{k = 0}^\infty  {\frac{1}{{( - n^3 )^k }}\left( {\frac{{n^2 }}{{6k + 1}} + \frac{{2n}}{{6k + 3}} + \frac{1}{{6k + 5}}} \right)} .
\]
In base $n^6$, length $12$, this is
\begin{equation}\label{eq.enca8mo}
\begin{split}
&n^2 \sqrt n \arctan \left( {\frac{{\sqrt n }}{n - 1}} \right)\\
&\qquad = \sum_{k = 0}^\infty  {\frac{1}{{n^{6k} }}\left( {\frac{{n^2 }}{{12k + 1}} + \frac{{2n}}{{12k + 3}} + \frac{1}{{12k + 5}} - \frac{{1/n}}{{12k + 7}} - \frac{{2/n^2 }}{{12k + 9}} - \frac{{1/n^3 }}{{12k + 11}}} \right)}.
\end{split}
\end{equation}
Identity~\eqref{eq.ulfkx09} follows upon setting $n=\alpha^{2r}$ in~\eqref{eq.enca8mo} and making use of~\eqref{eq.sbgg38c} and~\eqref{eq.yr9dvoi}. 

\end{proof}
\begin{example}
\begin{equation}\label{eq.idyup3r}
\frac{\pi }{4} = P(1,\alpha ^{12} ,12,( - \beta ,0, - 2\beta ^3 ,0, - \beta ^5 ,0,\beta ^7 ,0,2\beta ^9 ,0,\beta ^{11} ,0)),
\end{equation}
\begin{equation}
\arctan \left( {\frac{1}{4}} \right) = P(1,\alpha ^{12} ,12,( - \beta ^3 ,0, - 2\beta ^9 ,0, - \beta ^{15} ,0,\beta ^{21} ,0,2\beta ^{27} ,0,\beta ^{33} ,0)),
\end{equation}
\begin{equation}
\arctan \left( {\frac{1}{{\sqrt 5 }}} \right) = P(1,\alpha ^{12} ,12,(\beta ^2 ,0,2\beta ^6 ,0,\beta ^{10} ,0, - \beta ^{14} ,0, - 2\beta ^{18} ,0, - \beta ^{22} ,0)),
\end{equation}
\begin{equation}
\arctan \left( {\frac{1}{{3\sqrt 5 }}} \right) = P(1,\alpha ^{12} ,12,(\beta ^4 ,0,2\beta ^{12} ,0,\beta ^{20} ,0, - \beta ^{28} ,0, - 2\beta ^{36} ,0, - \beta ^{44} ,0)).
\end{equation}
\end{example}
\begin{remark}
Identity~\eqref{eq.idyup3r} is the same golden ratio base expansion of $\pi$ that was obtained in Theorem~\ref{thm.piaqnr3}.
\end{remark}

\begin{theorem}\label{thm.qiyqzuj}
If $r$ is an integer, then,
\begin{equation}\label{eq.sgqhdpw}
\sum_{k = 0}^{^\infty  } {\frac{1}{{\alpha ^{4rk} }}\left( {\frac{{2\beta ^r }}{{4k + 1}} - \frac{{2\beta ^{3r} }}{{4k + 3}}} \right)}  = \left\{ \begin{array}{l}
  - \arctan \left( {\frac{2}{{L_r }}} \right),\quad\mbox{$r$ odd}, \\ 
 \arctan \left( {\frac{2}{{F_r \sqrt 5 }}} \right),\quad\mbox{$r$ even}; \\ 
 \end{array} \right.
\end{equation}
that is
\begin{equation}\label{eq.ixlxa3s}
P(1,\alpha^{4r},4,(2\beta^r,0,-2\beta^{3r},0))= \left\{ \begin{array}{l}
  - \arctan \left( {\frac{2}{{L_r }}} \right),\quad\mbox{$r$ odd}, \\ 
 \arctan \left( {\frac{2}{{F_r \sqrt 5 }}} \right),\quad\mbox{$r$ even}. \\ 
 \end{array} \right.
\end{equation}
\end{theorem}
\begin{proof}
Setting $x=\beta^r$ in the identity
\[
2\arctan x = \arctan \left( {\frac{{2x}}{{1 - x^2 }}} \right)
\]
and using~\eqref{eq.sbgg38c} and~\eqref{eq.yr9dvoi}, we have
\begin{equation}\label{eq.utwtegd}
2\arctan\frac1{\alpha^r}=\left\{ \begin{array}{l}
  \arctan \left( {\frac{2}{{L_r }}} \right),\quad\mbox{$r$ odd}, \\ 
 \arctan \left( {\frac{2}{{F_r \sqrt 5 }}} \right),\quad\mbox{$r$ even}. \\ 
 \end{array} \right.
\end{equation}
Setting $n=\alpha^{2r}$ in~\eqref{eq.ubd4g1u} and comparing with~\eqref{eq.utwtegd}, we obtain~\eqref{eq.sgqhdpw}.
\end{proof}
\begin{example}
\begin{equation}\label{eq.ysjdgyu}
\arctan \frac{2}{{L_3 }} = \arctan \frac{1}{2} = P(1,\alpha ^{12} ,4,( - 2\beta ^3 ,0,2\beta ^9 ,0)),
\end{equation}
\begin{equation}
\arctan \left( {\frac{2}{{F_2 \sqrt 5 }}} \right) = \arctan \frac{2}{{\sqrt 5 }} = P(1,\alpha ^8 ,4,(2\beta ^2 ,0, - 2\beta ^6 ,0)).
\end{equation}
\end{example}
\section{Zero relations}
Zero relations are expansion formulas that evaluate to zero. They are useful in the determination and classification of new expansion formulas. A base~$\alpha$ expansion is not considered new if it can be written as a linear combination of existing formulas and known zero relations.
\subsection{Zero relations arising from the logarithm formulas}
\subsubsection*{Zero relation from $\log(F_3^2/L_3)=0$}
\begin{theorem}\label{thm.tstzqph}
We have
\[
\sum_{k = 0}^\infty  {\frac{1}{{\alpha ^{12k} }}\left( {\frac{1}{{6k + 1}} - \frac{{3\beta ^2 }}{{6k + 2}} - \frac{{8\beta ^4 }}{{6k + 3}} - \frac{{3\beta ^6 }}{{6k + 4}} + \frac{{\beta ^8 }}{{6k + 5}}} \right)}  = 0;
\]
that is,
\[
0 = P(1,\alpha ^{12} ,6,(1, - 3\beta ^2 , - 8\beta ^4 , - 3\beta ^6 ,\beta ^8 ,0)).
\]
\begin{proof}
We have
\begin{equation}\label{eq.k7a8mc6}
2\log F_3 -\log L_3=0.
\end{equation}
The expansion of $\log L_3$ given in~\eqref{eq.qhwege5} has the following base $\alpha ^{12}$, length $12$ version:
\begin{equation}\label{eq.bedkkp3}
\log L_3  = P(1,\alpha ^{12} ,6,(3\beta ^2 ,3\beta ^4 ,0,3\beta ^8 ,3\beta ^{10} ,0)).
\end{equation}
Use of~\eqref{eq.a4iqzr2} and~\eqref{eq.bedkkp3} in~\eqref{eq.k7a8mc6} yields the zero relation stated in Theorem~\ref{thm.tstzqph}.

\end{proof}
\end{theorem}

\subsubsection*{Zero relation from $\log(L_6/(F_4^2F_3))=0$}
\begin{theorem}
We have
\[
\begin{split}
0&=P(1,\alpha^{48},24,(1,-5\,\beta^{2},-2\,\beta^{4},3\,\beta^{6},\beta^{8},4\,\beta^{
10},\beta^{12},3\,\beta^{14},\\
&\qquad-2\,\beta^{16},-5\,\beta^{18},\beta^{20},0,\beta^{24}
,-5\,\beta^{26},-2\,\beta^{28},3\,\beta^{30},\beta^{32},\\
&\qquad\qquad 4\,\beta^{34},\beta^{36},3\,\beta^{38},-2\,\beta^{40},-5\,\beta^{42},\beta^{44},0)).
\end{split}
\]
\end{theorem}
\begin{proof}
Write $\log F_3$, $\log F_4$ and $\log L_6$, that is, identities~\eqref{eq.a4iqzr2}, \eqref{eq.osuk393} and~\eqref{eq.tncj4ik}, respectively, in the common base $\alpha^{48}$ and common length $24$ and use
\[
\log L_6 -2\log F_4 -\log F_3=0.
\]
\end{proof}
\subsubsection*{Zero relation from $\log(F_{12}/(F_3^4L_2^2))=0$}
\begin{theorem}
We have
\[
\begin{split}
0&=P(1,\alpha^{48},24,(1,-4\,\beta^{2},-5\,\beta^{4},0,\beta^{8},2\,\beta^{10},\beta^{12},0,-5\,\beta^{
16},-4\,\beta^{18},\beta^{20},0,\\
&\qquad\beta^{24},-4\,\beta^{26},-5\,\beta^{28},0,\beta^{
32},2\,\beta^{34},\beta^{36},0,-5\,\beta^{40},-4\,\beta^{42},\beta^{44},0)).
\end{split}
\]
\end{theorem}
\begin{proof}
Write $\log F_3$, $\log F_{12}$ and~$\log L_2$ from~~\eqref{eq.a4iqzr2}, \eqref{eq.jz4ufni} and~\eqref{eq.exslg0n}, in common base $\alpha^{48}$ and consider
\[
\log F_{12} - 4\,\log F_3 - 2\,\log L_2 = 0.
\]
\end{proof}
\subsection{Zero relations arising from the inverse tangent formulas}
\subsection*{Zero relation from $2\arctan(2/L_3) + \arctan(2/L_5) - \arctan(2/L_1)=0$}
\begin{theorem}\label{thm.q4lawys}
We have
\[
\begin{split}
0&=P(1,\alpha^{60},60,(1,0,-7\,\beta^{2},0,-4\,\beta^{4},0,-\beta^{6},0,7\,\beta^{8},0,-\beta^{10},0,{
b}^{12},0,-2\,\beta^{14},\\
&\qquad 0,\beta^{16},0,-\beta^{18},0,7\,\beta^{20},0,-\beta^{22
},0,-4\,\beta^{24},0,-7\,\beta^{26},0,\beta^{28},0,-\beta^{30},0,\\
&\qquad\,7\,\beta^{32},0,4\,\beta^{34},0,\beta^{36},0,-7\,\beta^{38},0,\beta^{40},0,-\beta^{42},0,2\,\beta
^{44},0,-\beta^{46},0,\\
&\qquad\;\beta^{48},0,-7\,\beta^{50},0,\beta^{52},0,4\,\beta^{54},0
,7\,\beta^{56},0,-\beta^{58},0)).
\end{split}
\]
\end{theorem}
\begin{proof}
Using the addition and subtraction formulas for inverse tangents, it is readily verified that
\[
\arctan \left( {\frac{2}{{L_1 }}} \right) - \arctan \left( {\frac{2}{{L_3 }}} \right)= \arctan \left( {\frac{3}{4}} \right)
\]
and
\[
\arctan \left( {\frac{2}{{L_5 }}} \right) + \arctan \left( {\frac{2}{{L_3 }}} \right)=\arctan \left( {\frac{3}{4}} \right);
\]
so that
\[
2\arctan \left( {\frac{2}{{L_3 }}} \right) + \arctan \left( {\frac{2}{{L_5 }}} \right) - \arctan \left( {\frac{2}{{L_1 }}} \right) = 0,
\]
from which the zero relation follows upon use of~\eqref{eq.ixlxa3s}.
\end{proof}
\begin{remark}
The zero relation stated in Theorem~\ref{thm.q4lawys} can also be obtained directly from
\[
\arctan\frac1{F_3}=\arctan\frac2{L_3},
\]
by writing~\eqref{eq.q0nwkrz} and~\eqref{eq.ysjdgyu} in the common base $\alpha^{60}$ and common length $60$; or from
\[
\arctan\frac1{F_4}=\arctan\frac1{L_2},
\]
using~\eqref{eq.u6gman9} and~\eqref{eq.blxa1pt}.
\end{remark}
\subsection*{Zero relation from\\ $2\arctan(1/L_1) - 2\arctan(2/(F_2\sqrt 5)) - \arctan(2/(F_6\sqrt 5))=0$}
\begin{theorem}
We have
\[
\begin{split}
0&=P(1,\alpha^{24},24,(1,4\,\beta,2\,\beta^{2},0,\beta^{4},2\,\beta^{5},-\beta^{6},0,-2\,\beta^{8},4\,\beta^
{9},-\beta^{10},0,\beta^{12},-4\,\beta^{13},\\
&\qquad2\,\beta^{14},0,\beta^{16},-2\,\beta^{17},-\beta^{18},0,-2\,\beta^{20},-4\,\beta^{21},-\beta^{22},0)).
\end{split}
\]
\end{theorem}
\begin{proof}
The identity
\[
\begin{split}
\frac{\pi }{2} - \arctan \left( {\frac{2}{{\sqrt 5 }}} \right) &= \arctan \left( {\frac{1}{{4\sqrt 5 }}} \right) + \arctan \left( {\frac{2}{{\sqrt 5 }}} \right)\\
& = \arctan \left( {\frac{{\sqrt 5 }}{2}} \right)
\end{split}
\]
can be arranged as
\[
2\arctan \left( {\frac{1}{{L_1 }}} \right) - 2\arctan \left( {\frac{2}{{F_2 \sqrt 5 }}} \right) - \arctan \left( {\frac{2}{{F_6 \sqrt 5 }}} \right) = 0
\]
which, on account of~\eqref{eq.ixlxa3s}, gives the stated zero relation.
\end{proof}
\section{Other degree $1$ base $\alpha$ expansions and zero relations}
\subsection*{Base $\alpha$ expansions of $\log\alpha$}
\begin{theorem}
\begin{equation}\label{eq.bn3dd2p}
\log\alpha=P(1,\alpha,2,(0,-\beta)).
\end{equation}
\end{theorem}
\begin{proof}
We have
\[
\log \alpha  = \frac12\Li_1 \left( {\frac1{\alpha }} \right) = \frac12\sum_{k = 0}^\infty  {\frac1{{\alpha ^k }}\;\frac{1/\alpha}{ k + 1}}  = \sum_{k = 0}^\infty  {\frac1{{\alpha ^k }}\;\frac{{ - \beta }}{{2k + 2}}} .
\]

\end{proof}
\begin{theorem}
\begin{equation}\label{eq.jd35my6}
\log\alpha=P(1,\alpha^2,2,(0,2\beta^2)).
\end{equation}
\end{theorem}
\begin{proof}
We have
\[
\log \alpha  = \Li_1 \left( {\frac{1}{\alpha ^2 }} \right) = \sum_{k = 0}^\infty  {\frac{1}{{\alpha ^{2k} }}\frac{{1/\alpha ^2 }}{{k + 1}}}  = \sum_{k = 0}^\infty  {\frac{1}{{\alpha ^{2k} }}\frac{{2\beta ^2 }}{{2k + 2}}} .
\]

\end{proof}
\subsection*{Another base~$\alpha$ expansion of $\log 2$}
\begin{theorem}
\begin{equation}\label{eq.kq97h2f}
\log 2=P(1,\alpha^3,3,(-\beta,\beta^2,2\beta^3)).
\end{equation}
\end{theorem}

\begin{proof}
A straightforward consequence of the identity
\[
\log 2=\Li_1\left(\frac1\alpha\right)-\Li_1\left(\frac1{\alpha^3}\right).
\]
\end{proof}

\subsection*{Another base~$\alpha$ expansion of $\log 5$}
\begin{theorem}
\begin{equation}\label{eq.i0gvzc3}
\log 5=P(1,\alpha^4,2,(4\,\beta^2,0)).
\end{equation}

\end{theorem}

\begin{proof}
A consequence of the identity
\[
\log 5=2\,\Li_1\left(\frac1{\alpha^2}\right) - 2\,\Li_1\left(-\frac1{\alpha^2}\right).
\]
\end{proof}
\subsection*{A length $2$, base~$\alpha$ zero relation}
\begin{theorem}
\begin{equation}\label{eq.eecbnrb}
P(1,\alpha^2,2,(1,3\,\beta))=0.
\end{equation}

\end{theorem}
\begin{proof}
Follows from
\[
\Li_1 \left( {\frac{1}{{\alpha ^2 }}} \right) + \Li_1 \left( { - \frac{1}{\alpha }} \right) = 0.
\]

\end{proof}
\begin{remark}
Relation~\eqref{eq.eecbnrb} also follows from~\eqref{eq.bn3dd2p} and~\eqref{eq.jd35my6}.
\end{remark}
\subsection*{A length $12$, base~$\alpha$ zero relation}
\begin{theorem}
\[
P(1,{\alpha}^{12},12,(1,\beta,-2\,{\beta}^{2},5\,{\beta}^{3},{\beta}^{4},10\,{\beta}^{5},{\beta}^{6}
,5\,{\beta}^{7},-2\,{\beta}^{8},{\beta}^{9},{\beta}^{10},2\,{\beta}^{11}))=0.
\]

\end{theorem}
\begin{proof}
Follows fron~\eqref{eq.a4iqzr2} and~\eqref{eq.kq97h2f}.
\end{proof}
\subsection*{A length $10$, base~$\alpha$ zero relation}
\begin{theorem}
\[
P(1,{\alpha}^{20},10,(1,-5\,{\beta}^{2},{\beta}^{4},-5\,{\beta}^{6},-4\,{\beta}^{8},-5\,{\beta}^
{10},{\beta}^{12},-5\,{\beta}^{14},{\beta}^{16},0))=0.
\]

\end{theorem}
\begin{proof}
Ensues from~\eqref{eq.iw04dtw} and~\eqref{eq.i0gvzc3}.
\end{proof}

\subsection*{A length $5$, base~$\alpha$ zero relation}
\begin{theorem}
\begin{equation}
P(1, \alpha^5, 5, (\beta, 1, -\beta, -\beta^4, -2\,\beta^4)) = 0.
\end{equation}
\end{theorem}
\begin{proof}
Setting $p=2\cos x$ in the identity
\[
\sum_{k = 1}^\infty  {\frac{{p^k \cos (kx)}}{k}}  =  - \frac{1}{2}\log (1 - 2p\cos x + p^2 )
\]
produces
\begin{equation}\label{eq.zjwrin2}
\sum_{k = 1}^\infty  {\frac{{(2\cos x)^k \cos (kx)}}{k}}  = 0.
\end{equation}
Now $2\cos(2\pi/5)=-\beta$.
\end{proof}
Thus, setting $x=2\pi/5$ in~\eqref{eq.zjwrin2} gives
\[
\sum_{k = 0}^\infty  {\frac{1}{{\alpha ^{5k} }}\left( {\frac{{\beta }}{{5k + 1}} + \frac1{{5k + 2}} - \frac{{\beta}}{{5k + 3}} - \frac{\beta^4}{{5k + 4}} - \frac{2\beta^4}{{5k + 5}}} \right)}  = 0;
\]
since
\[
\cos \left( {\frac{{2\pi }}{5}(5j - 4)} \right) = \frac{-\beta}2 = \cos \left( {\frac{{2\pi }}{5}(5j - 1)} \right),\quad j = 1,2, \ldots 
\]
and
\[
\cos \left( {\frac{{2\pi }}{5}(5j - 2)} \right) =  \frac1{2\beta} = \cos \left( {\frac{{2\pi }}{5}(5j - 5)} \right),\quad j = 1,2, \ldots 
\]

\end{document}